\documentclass[11pt,reqno]{article}

\usepackage{amsmath,amsthm,amsfonts,amssymb,latexsym,mathrsfs,color}
\usepackage{hyperref}

\newtheorem{theorem}{Theorem}
\newtheorem{corollary}[theorem]{Corollary}
\newtheorem{proposition}[theorem]{Proposition}

\newcommand{\msn}{{\mathcal S}_n}

\newcommand{\Stirling}[2]{\genfrac{\{}{\}}{0pt}{}{#1}{#2}}
\newcommand{\Eulerian}[2]{\genfrac{<}{>}{0pt}{}{#1}{#2}}

\title{Bessel polynomials, double factorials and context-free grammars\footnote{This work is supported by~NSFC (11126217) and the Fundamental Research Funds for the Central Universities (N100323013).}}

\author
{Shi-Mei Ma \footnote{ {\it Email address:}
shimeima@yahoo.com.cn (S.-M. Ma)} }
\date{\footnotesize School of Mathematics and Statistics,
        Northeastern University at Qinhuangdao,\\ Hebei 066004,
        China}

\begin{document}

\maketitle

\begin{abstract}
The purpose of this paper is to show that Bessel polynomials, factorials and Catalan triangle can be generated by using
context-free grammars.
\end{abstract}
\section{Introduction}\label{sec:intro}
The grammatical method was introduced by Chen~\cite{Chen93} in the study of exponential structures in combinatorics. Let $A$ be an alphabet whose letters are regarded as independent commutative indeterminates. Following Chen~\cite{Chen93}, a {\it context-free grammar} $G$ over $A$ is defined as a set of substitution rules replacing a letter in $A$ by a formal function over $A$.
The formal derivative $D$ is a linear operator defined with respect to a context-free grammar $G$.
For example, if $G=\{x\rightarrow xy, y\rightarrow y\}$, then $$D(x)=xy,D(y)=y,D^2(x)=x(y+y^2),D^3(x)=x(y+3y^2+y^3).$$
For any formal functions $u$ and $v$, we have
$$D(u+v)=D(u)+D(v),\quad D(uv)=D(u)v+uD(v) \quad and\quad D(f(u))=\frac{\partial f(u)}{\partial u}D(u),$$
where $f(x)$ is a analytic function.
Using Leibniz's formula, we obtain
\begin{equation}\label{Dnab-Leib}
D^n(uv)=\sum_{k=0}^n\binom{n}{k}D^k(u)D^{n-k}(v).
\end{equation}

Let $[n]=\{1,2,\ldots,n\}$. {\it The Stirling number of the second kind}
$\Stirling{n}{k}$ is the number of ways to partition $[n]$ into $k$ blocks.
Let $\msn$ denote the symmetric group of all permutations of $[n]$.
The {\it Eulerian number} $\Eulerian{n}{k}$ enumerates the number of permutations in $\msn$
with $k$ descents (i.e., $i<n,\pi(i)>\pi(i+1)$). The numbers $\Eulerian{n}{k}$
satisfy the recurrence relation
$$\Eulerian{n}{k}=(k+1)\Eulerian{n-1}{k}+(n-k)\Eulerian{n-1}{k-1},$$
with initial condition $\Eulerian{0}{0}=1$ and boundary conditions $\Eulerian{0}{k}=0$ for $k\geq 1$.
There is a close relationship between context-free grammars and combinatorics.
The reader is referred to~\cite{Chen121,Ma12} for recent results on this subject.
Let us now recall two classical results.
\begin{proposition}[{\cite[Eq. 4.8]{Chen93}}]
If $G=\{x\rightarrow xy, y\rightarrow y\}$,
then
\begin{equation*}
D^n(x)=x\sum_{k=1}^n\Stirling{n}{k}y^k.
\end{equation*}
\end{proposition}

\begin{proposition}[{\cite[Section 2.1]{Dumont96}}]
If $G=\{x\rightarrow xy, y\rightarrow xy\}$,
then
\begin{equation*}
D^n(x)=x\sum_{k=0}^{n-1}\Eulerian{n}{k}x^{k}y^{n-k}.
\end{equation*}
\end{proposition}

The purpose of this paper is to show that Bessel polynomials, factorials and Catalan triangle can be generated by using context-free grammars.
\section{Bessel polynomials}\label{sec:intro}
The well known {\it Bessel polynomials} $y_n(x)$ were introduced by Krall and Frink~\cite{Krall45}
as the polynomial solutions of the second-order differential equation
$$x^2\frac{d^2y_n(x)}{dx^2}+(2x+2)\frac{dy_n(x)}{dx}=n(n+1)y_n(x).$$
The Bessel polynomials $y_n(x)$ are a family of orthogonal polynomials and they have been extensively studied and applied (see~\cite[\textsf{A001498}]{Sloane}).
The polynomials $y_n(x)$ can be generated by using the Rodrigues formula
$$y_n(x)=\frac{1}{2^n}e^{\frac{2}{x}}\frac{d^n}{dx^n}\left(x^{2n}e^{-\frac{2}{x}}\right).$$
Explicitly, we have
$$y_n(x)=\sum_{k=0}^n\frac{(n+k)!}{(n-k)!k!}\left(\frac{x}{2}\right)^k.$$
These polynomials satisfy the recurrence relation
\begin{equation*}\label{ynx-recu}
y_{n+1}(x)=(2n+1)xy_n(x)+y_{n-1}(x)\quad {\text for}\quad n\geq 0,
\end{equation*}
with initial conditions $y_{-1}(x)=y_{0}(x)=1$.
The first few of the polynomials $y_n(x)$ are
\begin{align*}
 y_1(x)&=1+x, \\
 y_2(x)&=1+3x+3x^2, \\
 y_3(x)&=1+6x+15x^2+15x^3.
\end{align*}
We present here a grammatical characterization of the Bessel polynomials $y_n(x)$.
\begin{theorem}\label{Thm11}
If $G=\{a\rightarrow ab, b\rightarrow b^2c, c\rightarrow bc^2\}$, then
\begin{equation*}\label{Dnab}
D^n(ab)=ab^{n+1}y_n(c)\quad {\text for}\quad n\geq 0.
\end{equation*}
\end{theorem}
\begin{proof}
Let $$a(n,k)=\frac{(n+k)!}{2^k(n-k)!k!}.$$
Then $y_n(c)=\sum_{k=0}^na(n,k)c^k$.
It is easy to verify that
\begin{equation}\label{ank-recurrence}
a(n+1,k)=a(n,k)+(n+k)a(n,k-1).
\end{equation}
For $n\geq 0$,
we define
\begin{equation}\label{Dnab-def}
D^n(ab)=ab^{n+1}\sum_{k=0}^{n}E(n,k)c^k.
\end{equation}
Note that $D(ab)=ab^2(1+c)$. Hence $E(1,0)=a(1,0),E(1,1)=a(1,1)$.
It follows from~\eqref{Dnab-def} that
$$D^{n+1}(ab)=D(D^n(ab))=ab^{n+2}\sum_{k=0}^nE(n,k)c^k+ab^{n+2}\sum_{k=0}^n(n+k+1)E(n,k)c^{k+1}.$$
Therefore,
$$E(n+1,k)=E(n,k)+(n+k)E(n,k-1).$$
Comparing with~\eqref{ank-recurrence}, we see that
the coefficients $E(n,k)$ satisfy the same recurrence relation and initial conditions as $a(n,k)$, so they agree.
\end{proof}

For the context-free grammar $$G=\{a\rightarrow ab, b\rightarrow b^2c, c\rightarrow bc^2\},$$
in the same way as above we find that
$$D^n(a^2b)=2^na^2b^{n+1}y_n\left(\frac{c}{2}\right)\quad {\text for}\quad n\geq 0.$$
By Theorem~\ref{Thm11}, we obtain $D^k(a)=ab^ky_{k-1}(c)$ for $k\geq 0$. The {\it double factorial} of odd numbers are defined by
\begin{equation*}
(2n-1)!!=1 \cdot 3 \cdot 5 \cdot \dots \cdot (2n-1),
\end{equation*}
and for even numbers
\begin{equation*}
(2n)!!=2 \cdot 4 \cdot 6 \cdot \dots \cdot (2n).
\end{equation*}
As usual, set $(-1)!!=0!!=1$. It is clear that
\begin{equation*}\label{Dnb}
D^{n}(b)=(2n-1)!!b^{n+1}c^n\quad {\text for}\quad n\geq 0.
\end{equation*}

By~\eqref{Dnab-Leib}, the following corollary is immediate.
\begin{corollary}
For $n\geq 0$, we have
$$y_n(x)=\sum_{k=0}^n(2n-2k-1)!!\binom{n}{k}y_{k-1}(x)x^{n-k}.$$
\end{corollary}
\section{Polynomials associated with diagonal Pad\'e approximation to the exponential function}\label{approximation}
The Pad\'e approximations arise naturally in many branches of mathematics and have been extensively investigated (see~\cite{Prevost10} for instance).
The {\it diagonal Pad\'e approximation} to the exponential function $e^x$ is the unique rational function $$R_n(x)=\frac{P_n(x)}{P_n(-x)},$$
where $$P_n(x)=\sum_{k=0}^nM(n,k)x^{n-k} \quad {\text and}\quad M(n,k)=\frac{(n+k)!}{(n-k)!k!}.$$
Clearly, $P_n(1)=y_n(2)$, where $y_n(x)$ is the Bessel polynomials.
It is easy to verify that the numbers $M(n,k)$ satisfy the recurrence relation
\begin{equation}\label{Mnk-recu-1}
M(n+1,k)=M(n,k)+(2n+2k)M(n,k-1).
\end{equation}
The first few of the polynomials $P_n(x)$ are given as follows (see~\cite[A113025]{Sloane}):
\begin{align*}
  P_0(x)& =1, \\
  P_1(x)& =x+2, \\
  P_2(x)& =x^2+6x+12, \\
  P_3(x)& =x^3+12x^2+60x+120.
\end{align*}

We present here a grammatical characterization of the polynomials $P_n(x)$.
\begin{theorem}\label{thm8}
If $G=\{a\rightarrow ab^2, b\rightarrow b^3c^2, c\rightarrow b^2c^3\}$,
then
\begin{equation*}
D^n(ab^2)=ab^{2n+2}c^{2n}P_n\left(\frac{1}{c^2}\right).
\end{equation*}
\end{theorem}
\begin{proof}
For $n\geq0$, we define
\begin{equation}\label{Nnk-recu-2}
D^n(ab^2)=ab^{2n+2}\sum_{k=0}^nN(n,k)c^{2k}.
\end{equation}
Note that $D(ab^2)=ab^4(1+2c^2)$.
Hence $N(1,0)=M(1,0),N(1,1)=M(1,1)$.
It follows from~\eqref{Nnk-recu-2} that
$$D^{n+1}(ab^2)=ab^{2n+4}\sum_{k=0}^nN(n,k)c^{2k}+ab^{2n+4}\sum_{k=0}^n(2n+2k+2)N(n,k)c^{2k+2}.$$
Therefore,
$$N(n+1,k)=N(n,k)+(2n+2k)N(n,k-1).$$
Comparing with~\eqref{Mnk-recu-1}, we see that
the coefficients $N(n,k)$ satisfy the same recurrence relation and initial conditions as $M(n,k)$, so they agree.
\end{proof}

Along the same lines, we immediately deduce the following corollary.
\begin{corollary}
Let $y_n(x)$ be the Bessel polynomials.
If $G=\{a\rightarrow ab^2, b\rightarrow b^3c^2, c\rightarrow b^2c^3\}$,
then
\begin{equation*}
D^n(a^2b^2)=2^na^2b^{2n+2}y_n(c^2).
\end{equation*}
\end{corollary}
\section{Double factorials}\label{sec:factorial}
The following identity was studied systematically by
Callan~\cite[Section 4.8]{Callan}:
\begin{equation}\label{Callan}
\sum_{k=1}^nk!\binom{2n-k-1}{k-1}(2n-2k-1)!!=(2n-1)!!.
\end{equation}
As pointed out by Callan~\cite{Callan}, the identity~\eqref{Callan} counts different combinatorial structures, such as  {\it increasing ordered trees} of $n$ edges by outdegree $k$ of the root and the sum of the weights of all vertices labeled $k$ at depth $n-1$ in the {\it Catalan tree} (see~\cite[\textsf{A102625}]{Sloane}).

Let $$R(n,k)=k!\binom{2n-k-1}{k-1}(2n-2k-1)!!.$$
Thus, $\sum_{k=1}^nR(n,k)=(2n-1)!!$. It is easy to verify that
\begin{equation}\label{Tnk-recu}
R(n+1,k)=(2n-k)R(n,k)+kR(n,k-1),
\end{equation}
with initial conditions $R(0,0)=1$ and $R(0,k)=0$ for $k\geq 1$ or $k<0$. For $n\geq 1$,
let $R_n(x)=\sum_{k=1}^nR(n,k)x^k$.
The first few of the polynomials $R_n(x)$ are
\begin{align*}
  R_1(x)& =x, \\
  R_2(x)& =x+2x^2, \\
  R_3(x)& =3x+6x^2+6x^3,\\
  R_4(x)& =15x+30x^2+36x^3+24x^4.
\end{align*}

\begin{theorem}\label{Thm5}
If $G=\{a\rightarrow a^2b, b\rightarrow b^2c, c\rightarrow bc^2\}$, then
\begin{equation}\label{Dnab-3}
D^n(a)=ab^{n}\sum_{k=1}^{n}R(n,k)a^{k}c^{n-k}\quad {\text for}\quad n\geq 0.
\end{equation}
\end{theorem}
\begin{proof}
Note that $D(a)=a^2b$ and $D^2(a)=ab^2(ac+2a^2)$.
For $n\geq 1$, we define
\begin{equation}\label{Dna-def}
D^n(a)=ab^{n}\sum_{k=1}^{n}r(n,k)a^{k}c^{n-k}.
\end{equation}
Hence $r(1,1)=R(1,1),r(2,1)=R(2,1)$ and $r(2,2)=R(2,2)$.
It follows from~\eqref{Dna-def} that
$$D(D^n(a))=ab^{n+1}\sum_{k=1}^n(2n-k)r(n,k)a^kc^{n-k+1}+ab^{n+1}\sum_{k=0}^n(k+1)r(n,k)a^{k+1}c^{n-k}.$$
Therefore,
$$r(n+1,k)=(2n-k)r(n,k)+kr(n,k-1).$$
Comparing with~\eqref{Tnk-recu}, we see that
the coefficients $r(n,k)$ satisfy the same recurrence relation and initial conditions as $R(n,k)$, so they agree.
\end{proof}

In the following discussion, we also consider the context-free grammar
\begin{equation*}\label{grammar}
G=\{a\rightarrow a^2b, b\rightarrow b^2c, c\rightarrow bc^2\}.
\end{equation*}
Note that $$D(ab)=a^2b^2+ab^2c,D^2(ab)=ab^3(3c^2+3ac+2a^2).$$
For $n\geq 0$, we define
\begin{equation}\label{Hnk-def}
D^n(ab)=ab^{n+1}\sum_{k=0}^nH(n,k)a^kc^{n-k}.
\end{equation}
It follows that
$$D(D^n(ab))=ab^{n+2}\sum_{k=0}^n(2n-k+1)H(n,k)a^kc^{n-k+1}+ab^{n+2}\sum_{k=0}^n(k+1)H(n,k)a^{k+1}c^{n-k}.$$
Hence, the numbers $H(n,k)$ satisfy the recurrence relation
\begin{equation}\label{Gnk-recu}
H(n+1,k)=(2n-k+1)H(n,k)+kH(n,k-1),
\end{equation}
with initial conditions $H(1,0)=H(1,1)=1$ and $H(1,k)=0$ for $k\geq 2$ or $k<0$.
Using~\eqref{Gnk-recu}, it is easy to verify that
$$H(n,k)=\frac{(2n-k)!}{2^{n-k}(n-k)!}.$$
It should be noted that the numbers $H(n,k)$ are entries in a {\it double factorial triangle} (see~\cite[\textsf{A193229}]{Sloane}).
In particular, we have $H(n,0)=(2n-1)!!, H(n,n)=n!$ and $\sum_{k=0}^nH(n,k)=(2n)!!$.
Moreover, combining~\eqref{Dnab-Leib}, \eqref{Dnab-3} and~\eqref{Hnk-def}, we obtain
$$H(n,k)=\sum_{m=k}^n\binom{n}{m}(2n-2m-1)!!R(m,k)$$
for $n\geq 1$ and $0\leq k\leq n$.

For $n\geq 1$, we define
$$x(x+2)(x+4)\cdots (x+2n-2)=\sum_{k=1}^np(n,k)x^k$$
and
$$(x+1)(x+3)\cdots (x+2n-1)=\sum_{k=0}^nq(n,k)x^k.$$
The the triangular arrays $\{p({n,k})\}_{n\geq 1,1\leq k\leq n}$ and $\{q({n,k})\}_{n\geq 1,0\leq k\leq n}$ are both {\it double Pochhammer triangles} (see~\cite[\textsf{A039683,A028338}]{Sloane}).
The following theorem is in a sense ``dual" to Theorem~\ref{Thm5},
and we omit the proof for brevity.
\begin{theorem}\label{thmpq}
If $G=\{a\rightarrow ab^2, b\rightarrow b^2c, c\rightarrow bc^2\}$,
then we have
$$D^n(a)=ab^n\sum_{k=1}^np(n,k)b^kc^{n-k}$$
and
$$D^n(ab)=ab^{n+1}\sum_{k=0}^nq(n,k)b^kc^{n-k}.$$
\end{theorem}

Set $p(0,0)=q(0,0)=1$. By~\eqref{Dnab-Leib},
we immediately obtain
$$q(n,k)=\sum_{m=k}^n\binom{n}{m}(2n-2m-1)!!p(m,k)$$
for $n\geq 0$ and $0\leq k\leq n$.

\section{Catalan triangle}\label{sec:Catalan's triangle}
The classical {\it Catalan triangle}
is defined by the recurrence relation
$$T(n,k)=T(n-1,k)+T(n,k-1),$$
with initial conditions $T(0,0)=1$ and $T(0,k)=0$ for $k>0$ or $k<0$ (see~\cite[\textsf{A009766}]{Sloane}).
The numbers $T(n,k)$ are often called {\it ballot numbers}.  Explicitly,
\begin{equation}\label{Catalan}
T(n,k)=\binom{n+k}{k}\frac{n-k+1}{n+1}\quad {\text for}\quad 0\leq k\leq n.
\end{equation}
Moreover, $\sum_{k=0}^{n}T(n,k)=T(n+1,n+1)=C(n+1)$, where $C(n)$ is the  well known {\it Catalan number}.
Catalan numbers appear in a wide range of problems (see~\cite{Sagan12} for instance).

It follows from~\eqref{Catalan} that
\begin{equation}\label{Catalan-recu}
(n+2)T(n+1,k)=(n-k+2)T(n,k)+(2n+2k)T(n,k-1).
\end{equation}
This recurrence relation gives rise to the following result.

\begin{theorem}
If $G=\{a\rightarrow a^2b^2, b\rightarrow b^3c^2, c\rightarrow b^2c^3\}$,
then we have
\begin{equation*}\label{Catalan-Dnab}
D^n(a^2b^2)=(n+1)!a^2b^{2n+2}\sum_{k=0}^nT(n,k)a^{n-k}c^{2k}.
\end{equation*}
\end{theorem}
\begin{proof}
It is easy to verify that $D(a^2b^2)=2a^2b^4(a+c^2)$ and $D^2(a^2b^2)=3!a^2b^6(a^2+2ac^2+2c^4)$.
For $n\geq 0$, we define
\begin{equation*}\label{tnk}
D^n(a^2b^2)=(n+1)!a^2b^{2n+2}\sum_{k=0}^nt(n,k)a^{n-k}c^{2k}.
\end{equation*}
Note that
$$\frac{D^{n+1}(a^2b^2)}{(n+1)!a^2b^{2n+4}}=\sum_{k=0}^n(n-k+2)t(n,k)a^{n-k+1}c^{2k}+\sum_{k=0}^n(2n+2k+2)t(n,k)a^{n-k}c^{2k+2}.$$
Thus, we get
$$(n+2)t(n+1,k)=(n-k+2)t(n,k)+(2n+2k)t(n,k-1).$$
Comparing with~\eqref{Catalan-recu}, we see that
the coefficients $t(n,k)$ satisfy the same recurrence relation and initial conditions as $T(n,k)$, so they agree.
\end{proof}

In the same way as above we find that if $G=\{a\rightarrow a^2b^2, b\rightarrow b^3c^2, c\rightarrow b^2c^3\}$, then
\begin{equation*}\label{Catalan-Dnab}
D^n(ab^2)=n!ab^{2n+2}\sum_{k=0}^n\binom{n+k}{k}a^{n-k}c^{2k}
\end{equation*}
and
\begin{equation*}\label{Catalan-Dnab}
D^n(b)=\prod_{k=0}^{n-1}(4k+1)b^{2n+1}c^{2n}.
\end{equation*}
It should be noted that $\binom{n+k}{k}$ is the number of lattice paths from $(0,0)$ to $(n,k)$ using steps $(1,0)$ and $(0,1)$ (see~\cite[\textsf{A046899}]{Sloane}) and $\prod_{k=0}^{n-1}(4k+1)$ is the {\it quartic factorial number} (see~\cite[A007696]{Sloane}).



\end{document}